\documentclass{amsart}
\usepackage{amssymb}
\usepackage{amsmath}
\usepackage{graphicx}

\newtheorem{thm}{Theorem}[section]
\newtheorem{cor}[thm]{Corollary}
\newtheorem{lem}[thm]{Lemma}
\newtheorem{prop}[thm]{Proposition}

\newtheorem{quest}{Question}
\newtheorem{fact}[thm]{Fact}
\newtheorem{exam}[thm]{Example}
\theoremstyle{definition}
\newtheorem{defn}{Definition}
\theoremstyle{remark}
\newtheorem{rem}{Remark}
\newcommand{\Sum}{\bigcup}

\newcommand{\sm}{\setminus}
\newcommand{\sbe}{\subseteq}
\newcommand{\ex}{\exists}
\newcommand{\fa}{\forall}

\newcommand{\es}{\emptyset}
\renewcommand{\phi}{\varphi}

\newcommand{\sid}{\sigma\mbox{-ideal}}
\newcommand{\sidi}{\sigma\!\mbox{-}ideal}
\newcommand{\cont}{{\mathfrak c}}

\DeclareMathOperator{\add}{add}
\DeclareMathOperator{\non}{non}
\DeclareMathOperator{\cf}{cf}
\DeclareMathOperator{\Perf}{Perf}
\DeclareMathOperator{\Borel}{Borel}
\DeclareMathOperator{\Open}{Open}

\newcommand{\set}[1]{\{#1\}}

\renewcommand{\AA}{{\mathcal A}}
\newcommand{\BB}{{\mathcal B}}
\newcommand{\DD}{{\mathcal D}}
\newcommand{\II}{{\mathcal I}}
\newcommand{\NN}{{\mathcal N}}
\newcommand{\PP}{{\mathcal P}}

\newcommand{\bbr}{\mathbb{R}}
\newcommand{\bbz}{\mathbb{Z}}
\newcommand{\bbq}{\mathbb{Q}}
\newcommand{\bbk}{\mathbb{K}}
\newcommand{\bbl}{\mathbb{L}}

\newcommand{\Zeberski}{{\.Z}eberski\ }
\newcommand{\Ralowski}{Ra{\l}owski\ }

\begin{document}
\title{Bernstein sets and $\kappa$-coverings}
\author[J. Kraszewski, R. \Ralowski, \;P. Szczepaniak and S. \Zeberski]
       {Jan Kraszewski, Robert Ra{\l}owski,
        Przemys{\l}aw Szczepaniak and  Szymon \Zeberski}
\email[Jan Kraszewski]{jan.kraszewski@math.uni.wroc.pl}
\email[Robert Ra{\l}owski]{robert.ralowski@pwr.wroc.pl}
\email[Przemys{\l}aw Szczepaniak]{pszczepaniak@math.uni.opole.pl}
\email[Szymon \Zeberski]{szymon.zeberski@pwr.wroc.pl}
\address{Jan Kraszewski, Mathematical Institute, University of Wroc{\l}aw,
         pl. Grunwaldzki 2/4, 50-384 Wroc{\l}aw, Poland.}
\address{Robert \Ralowski and Szymon \Zeberski, Institute of
         Mathematics and Computer Sciences,
         Wroc{\l}aw University of Technology, Wybrze\.ze Wyspia\'n\-skie\-go 27,
         50-370 Wroc{\l}aw, Poland.}
\address{Przemys{\l}aw Szczepaniak , Institute of
         Mathematics and Computer Sciences,
	 University of Opole,
         ul. Oleska 48, 45-052 Opole, Poland.}

\maketitle

\begin{abstract}
In this paper we study a notion of a $\kappa$-covering in connection with Bernstein sets and other types of
nonmeasurability. Our results correspond to those obtained by Muthuvel in \cite{muthuvel} and Nowik in
\cite{nowik}. We consider also other types of coverings.
\end{abstract}

\section{Definitions and notation}

In 1993 Carlson in his paper \cite{Carlson} introduced a notion of $\kappa$-coverings and used it for investigating
whether some ideals are or are not $\kappa$-translatable. Later on $\kappa$-coverings were studied by other
authors, e.g. Muthuvel (cf. \cite{muthuvel}) and Nowik (cf. \cite{nowik}, \cite{nowik2}). In this paper we
present new results on $\kappa$-coverings in connection with Bernstein sets. We also introduce two natural
generalizations of the notion of $\kappa$-coverings, namely $\kappa$-S-coverings and $\kappa$-I-coverings.

We use standard set-theoretical notation and terminology from \cite{bartosz}. Recall that the cardinality of the set
of  all real numbers $\bbr$ is denoted by $\cont$. The cardinality of a set $A$ is denoted by $|A|$. If $\kappa$ is a
cardinal number then
$$
\begin{aligned}
  &[A]^\kappa=\{ B\sbe A:\; |B|=\kappa\};\\
  &[A]^{<\kappa}=\{ B\sbe A:\; |B|<\kappa\}.
\end{aligned}
$$
The cofinality of $\kappa$ is denoted by $\cf(\kappa)$. The power set of a set $A$ is denoted by $\PP(A)$.

For a given uncountable abelian Polish group $(X,+)$, the family of all uncountable perfect subsets of $X$ is denoted
by $\Perf(X),$ 
the family of all open subsets of $X$ is denoted
by $\Open(X)$
and the family of all Borel subsets of $X$ is denoted by $\Borel(X)$. We say that a set $B\sbe X$ is
\emph{a Bernstein set} if for every uncountable set $Z\in\Borel(X)$ both sets $Z\cap B$ and $Z\sm B$ are
nonempty.

In this paper $\II$ stands for a $\sid$ of subsets of a given uncountable abelian Polish group $(X,+)$. We will always
assume that $\II$ is proper and group invariant, contains singletons and has a Borel base (i.e. for every set
$A\in \II$ we can find a Borel set $B\in\II$ such that $A\sbe B$). We will use two cardinal characteristics of an
ideal $\II$: the additivity number $\add(\II)$ and the uniformity number $\non(\II)$, defined as follows:
$$
\begin{aligned}
  \add(\II)&=\min\set{|\AA|\,:\,A\sbe\II\,\land\,\Sum A\notin\II};\\
  \non(\II)&=\min\set{|A|\,:\,A\sbe X\,\land\,A\notin\II}.
\end{aligned}
$$
\indent
Let us recall the notion investigated for instance in \cite{CMRRZ}.

\begin{defn}
Let $N\sbe X$. We say that the set $N$ is \emph{completely $\II$-nonmeasurable} if
  $$(\fa A\in\Borel(X)\sm\II)(A\cap N \notin \II\,\land\,A\cap(X\sm N)\notin\II).$$
\end{defn}

In particular, for the $\sid$ of Lebesgue null sets $\NN\sbe\PP(\bbr)$ we have that a set $N\sbe\bbr$ is completely
$\NN$-nonmeasurable if and only if the inner measure of $N$ and the inner measure of the complement of $N$ are zero.
One can observe that if $\II$ is a $\sid$ of our interest (i.e. having properties mentioned above) then every Bernstein
set is completely $\II$-nonmeasurable. Hence the notion of a completely $\II$-nonmeasurable set generalizes the notion
of a Bernstein set.

\begin{defn}[Polish ideal space] We say that the pair $(X,\II)$ is Polish ideal space iff $X$ is uncountable Polish space and $\II\subseteq P(X)$ is a $\sigma$ ideal with singletons and Borel base.
\end{defn}

\begin{defn}[Polish ideal group] We say that the triple $(X,\II,+)$ is Polish ideal group iff $(X,\II)$ is a Polish ideal space space, $(X,+)$ is a Abelian group and $\II$ is invariant under group action $+$ which means that $(\forall A\in \II)(\forall t\in X)\; t+A\in \II$.
\end{defn}

While constructing completely $\II$-nonmeasurable sets having interesting covering properties we will concentrate
on $\sid$s including all unit spheres. Let us observe that classical $\sid$s such as the $\sid$ of null sets and
the $\sid$ of meager sets have this property.

The following notion of a tiny set is very useful in recursive constructions of completely $\II$-nonmeasurable sets.

\begin{defn}
Let $(X,\II,+)$ be Polish ideal group and let us fix a family $\AA\sbe\II$. We say that a perfect set $P\in\Perf(X)$ is \emph{a tiny set with respect
to $\AA$} if
\begin{enumerate}
 \item $(\fa t\in X)(\fa A\in\AA)\,|(P+t)\cap A|\le\omega,$
 \item $(\forall B\in \Borel(X)\setminus \II)(\exists t\in X)\;\; |(P+t)\cap B|=\cont.$
\end{enumerate}
\end{defn}
\noindent

In \cite{ralowski} \Ralowski proved the following useful lemma.

\begin{lem}\label{covering}
Let $\AA\sbe\II$. If there exists a perfect set $P\in\Perf(X)$, which is tiny with respect to $\AA$ then
  $$\min\set{|\BB|:\BB\sbe\AA\land (\ex B\in\Borel(X)\sm\II)( B\sbe\Sum\BB)}=\cont.$$
\end{lem}

\begin{defn}[Steinhaus property] Let $(X,\II,+)$ be any Polish ideal group then $\sigma$ ideal $\II$ has Steinhaus property iff
$$
(\forall A,B\in \Borel(X)\sm\II)(\exists U\in\Open(X))\;\; \emptyset\neq U\subseteq A-B.
$$
\end{defn}
It is well known that ideals of all meager sets $\bbk$ and Lebesgue null sets $\bbl$ has a Steinhaus property.

Let observe that the following fact is true.
\begin{fact}\label{stein} Let $(X,\II,+)$ is a Polish ideal group and $\II$ has Steinhaus property. Let us consider any Borel $\II$ positive set $B$ (i.e. $B\in\Borel(X)\sm\II$) and let $Q\subseteq X$ be any dense countable subgroup of $X$. Then $(B+Q)^c\in\II$.
\end{fact}
\begin{proof} Let $B^*=B+Q.$ It is a Borel set. Assume that $(B^*)^c\notin \II$ then there exists a Borel $\II$ positive set $A\in \Borel(X)\sm\II$ such that $A\cap B^*=\emptyset$. But by Steinhaus property there exists nonempty open set $U\subseteq X$ such that $U\subseteq A-B^*$. Then there exists  some $q\in Q$ and $b\in B^*$ such that $q+b\in A.$ Since $Q+B^*=B^*,\; q+b\in B^*\cap A,$ what gives a contradiction.
\end{proof}

Now we will concentrate on ideals $\bbl,\ \bbk.$ Next lemma is probably folclore, but for reader's convinience we give a proof of if (made by Cicho{\'n}).
\begin{lem}\label{cichon} Let $\II$ be one of them ideals $\bbl,\bbk$ then
$$
(\forall B\in \Borel(\bbr))(\forall P\in \Perf(\bbr))(\exists t\in\bbr)\;\; |(t+P)\cap B|=\cont.
$$
\end{lem}
\begin{proof}(Cicho{\'n}) Firstly, let us assume that $\omega_1<cov(\II)$. Now choose any subset $T\in [P]^{\omega_1}$ of the perfect set $P$ with cardinality $\omega_1$. Let $B^*=B+\bbq.$ Then using assumption $\omega_1<cov(\II)$  and Fact \ref{stein} we have $\bigcup_{t\in T} (t+B^*)^c\ne \bbr$. Then $\bigcup_{t\in T}t+B^*$ is nonempty set. Let $y\in\bigcap_{t\in T} (t+B^*)$ be any its element. Then by simple calculation we have
\begin{align*}
y\in\bigcap_{t\in T}(-t)+B^*\iff &(\forall t\in T) y\in -t+B^*\iff (\forall t\in T)y+t\in B^* \\
\iff &(\forall t\in T)t\in -y+B^*\iff T\subseteq -y+B^*.
 \end{align*}
But $\bbq$ is countable, so there are $x\in\bbr$ and $S\in [T]^{\omega_1}$ such that $S\subseteq x+B.$ But $S\subseteq P$ and $P$ is perfect set, so $|(x+B)\cap P|=\cont$.

Now let $V$ be any model of $ZFC$ theory. There is a generic extension $V[G]$ fulfilling condition $MA+\neg \cont=\omega_1.$ So, in $V[G]\;  \omega_1<cov(\II)$. But the following formula
$$
(\forall P\in\Perf)(\forall B\in \Borel(\bbr)\sm\II)(\exists x\in\bbr)\; |(x+P)\cap B|=\cont.
$$
is $\Sigma_2^1.$ So by Shoenfield absolutness theorem (cf. \cite{shoenfield}) it  holds also in ground model $V$.
\end{proof}

\begin{rem} Another proof for measure case was given by Ryll-Nardzewski. His proof is based on convolution measures. The other proof was due to Morayne, where density point of measure was used.
\end{rem}

\begin{rem} Let us observe that Lemma \ref{cichon} is true whenever we replace ideal of meager or null sets by any $\sidi$ ideal $\II$ with Borel base and with the Steinhaus property for which it is consistent that $\omega_1<cov(\II)$ and the Borel codes for a sets from ideal $\II$ are absolute between transitive models of ZFC theory.
\end{rem}

Lemma \ref{cichon} gives simpler characterization of tiny set in case $\II=\bbl,\bbk$ namely we have the following corollary.
\begin{cor} If $\II\in\{ \bbl,\bbk\}$ then $A$ is tiny set with respect
to $\AA$ iff $$(\fa t\in X)(\fa A\in\AA)\,|(P+t)\cap A|\le\omega.$$
\end{cor}

Let us notice that previous result is not true in general (as pointed by referee). Namely we have the following example.
\begin{exam}[given by referee] Assume that the cofinality of the meager ideal on $\bbr$ is $\omega_1$ and $\omega_1<\cont$. Let $( A_\alpha:\alpha<\omega_1)$ be a cofinal tower of meager sets in $\bbr$, $X=\bbr\times\bbr$. Let $\II$ be the $\sigma$-ideal of subsets of $X$ with meager projection on the first coordinate. Let $\AA=\{ A_\alpha\times\{ 0\}:\alpha<\omega_1\}$, $P=\{0\}\times\bbr$ and $B=\bbr\times\{0\}$. Then $B\in \Borel(X)\sm\II$, $|P\cap A_\alpha\times\{0\}|\le 1$ and $B\subseteq \bigcup \AA$.
\end{exam}

%%%%%%%%%%%%%%%%%
In our applications we will concentrate on families of unit spheres in $\bbr^n$.

\begin{lem}\label{covering2}
Let $\II\in\{ \bbk,\bbl\}$. Let $\DD$ be a family of unit spheres of size less than continuum and let $B\in\Borel(\bbr^n)\sm\II$. Then
  $$\big|B\sm\Sum\DD\big|=\cont.$$
\end{lem}
\begin{proof}
Observe that every line is a tiny set with respect to the family of all unit spheres. So according to
Lemma~\ref{covering} the set $B$ cannot be covered by $\Sum\DD$. Hence $|B\sm\Sum\DD|=\cont$.
\end{proof}

\section{Coverings on the real line}

In \cite{Carlson} Carlson introduced the following definition.

\begin{defn}\label{covdef}
We say that the set $A\sbe\bbr$ is \emph{a $\kappa$-covering} if for every set $B\sbe\bbr$ of cardinality $\kappa$
there exists a real number $x\in\bbr$ such that $B+x\sbe A.$
\end{defn}

Analogously, a set $A\sbe\bbr$ is \emph{a $\,<\!\kappa$-covering} if every set $B\sbe\bbr$ of cardinality less then
$\kappa$ can be translated into it (cf. \cite{muthuvel}). Of course, these definitions are reasonable also for other
uncountable abelian Polish groups.

Nowik in his papers studied partitions of the Cantor space $2^\omega$ into regular (Borel) $\omega$-coverings. He
constructed such a partition of size continuum (\cite{nowik}) and a partition into two sets, one $F_\sigma$, one
$G_\delta$, having some special property. We present analogous and even stronger results concerning irregular
(Bernstein) sets.

First we prove that we can find a partition of the real line into two Bernstein sets having no covering properties.

\begin{thm}
 There exists a partition of the real line $\bbr$ into two sets $A,\ B$ such that each of them is
 a Bernstein set and none of them is a $2$-covering.
\end{thm}

\begin{proof}
  Let $\Perf(\bbr)=\set{P_\alpha:\alpha<\cont}$ and $\bbr=\set{r_\alpha:\alpha<\cont}$ be fixed enumerations
  of all perfect subsets of the real line and of the reals, respectively.
  By transfinite induction we build two increasing sequences $(A_\alpha)_{\alpha<\cont},\ (B_\alpha)_{\alpha<\cont}$
  of subsets of $\bbr$ such that for every $\alpha<\cont$ the following conditions are satisfied:
  \begin{enumerate}
   \item $|A_\alpha |=|B_\alpha |=|\alpha |;$
   \item $r_\alpha\in A_\alpha\cup B_\alpha;$
   \item $A_\alpha\cap P_\alpha\neq\emptyset,\ B_\alpha\cap P_\alpha\neq\emptyset;$
   \item $A_\alpha\cap B_\alpha=\emptyset.$
  \end{enumerate}
  Moreover, to ensure that $A_\alpha$ and $B_\alpha$ are not $2$-coverings we want them to satisfy two more
  conditions:
 \begin{enumerate}
  \item[(5)] $(\forall x\in A_\alpha)(\{x-1,\ x+1\}\subseteq B_\alpha)$;
  \item[(6)] $(\forall x\in B_\alpha)(\{x-1,\ x+1\}\subseteq A_\alpha)$.
 \end{enumerate}
  Now, the set $\set{0,1}$ cannot be translated neither into $A_\alpha$ nor into $B_\alpha$.

  We are able to fulfill all these conditions because being at the $\alpha$th step of our construction we
  know that $|\Sum_{\beta<\alpha}(A_\beta\cup B_\beta)|<\cont$ and for every $\beta<\alpha$ we have
  $(A_\beta\cup B_\beta)+\bbz=A_\beta\cup B_\beta$.

 Finally, we put $A=\Sum_{\alpha<\cont}A_\alpha$ and $B=\bigcup_{\alpha<\cont}B_\alpha$. These sets are Bernstein
 sets because of (3), form a partition of $\bbr$ because of (2) and (4) and are not $2$-coverings as neither are
 sets $A_\alpha$ and $B_\alpha$.
\end{proof}

The next theorem is in contrast with the previous one.

\begin{thm}\label{cofi}
 There is a partition $\set{B_\xi:\xi<\cont}$ of the real line into Bernstein sets such
 that for every $\xi<\cont$ the set $B_\xi$ is a $\,<\!cf(\cont)$-covering.
\end{thm}
\begin{proof}
  Let $\kappa=cf(\cont)$ and let $(c_\alpha)_{\alpha<\kappa}$ be a cofinal increasing sequence of elements of $\cont$.
   Let us fix an increasing sequence $(R_\alpha)_{\alpha<\kappa}$ of subsets of $\bbr$ and a sequence
   $(\PP_\alpha)_{\alpha<\kappa}$ of families of perfect subsets of $\bbr$ such that
  $$\bbr=\bigcup_{\alpha<\kappa}R_\alpha,\quad\Perf(\bbr)=\bigcup_{\alpha <\kappa}\PP_\alpha$$
  and $|R_\alpha|=|\PP_\alpha|=|c_\alpha|$.

  By transfinite induction we build a sequence of families $(\set{B^\alpha_\xi:\xi<c_\alpha})_{\alpha<\kappa}$
  satisfying the following conditions:
  \begin{enumerate}
    \item for every $\alpha<\kappa$ and for every $\xi< c_\alpha$ we have $|B_\xi^\alpha|=|c_\alpha|$;
    \item for every $\alpha<\kappa$ sets from the family $\set{B^\alpha_\xi:\xi<c_\alpha}$ are pairwise disjoint;
    \item for every $\xi<\cont$ and every $\alpha_1<\alpha_2<\kappa$ such that $\xi<c_{\alpha_1}$ we have
          $B^{\alpha_1}_\xi\sbe B^{\alpha_2}_\xi$;
    \item for every $\alpha<\kappa$ the intersection $B_\xi^\alpha\cap P$ is nonempty for
          every $\xi<c_\alpha$ and every perfect set $P$ from the family $\PP_\alpha$;
    \item for every $\alpha<\kappa$ and every $\xi<c_\alpha$ there exists $x\in\bbr$ such that
          $x+R_\alpha\sbe B^\alpha_\xi$.
  \end{enumerate}
 We obtain such a sequence as follows. Assume that we are at the $\alpha$th step of the construction, so we have
 already built families $\set{B^\beta_\xi:\xi<c_\beta}$ for $\beta<\alpha$. One can observe that the cardinality
 of the union of all sets $B^\beta_\xi$ constructed so far (let us notice this sum by $S$) is small:
  $$|S|=\left|\Sum_{\beta<\alpha}\Sum_{\xi<c_\beta}B^\beta_\xi\right|\le|c_\alpha|\cdot|c_\alpha|\cdot|\alpha|=
    |c_\alpha|<\cont.$$
 For every $\xi<c_\alpha$ let us put
  $$B^{<\alpha}_\xi=\Sum_{\beta<\alpha}B^\beta_\xi$$
 (the set $B^{<\alpha}_\xi$ is empty for $\Sum_{\beta<\alpha}c_\beta\le\xi<c_\alpha$). Let us notice that
 there are at most $c_\alpha$ many real numbers $x$ such that $(x+R_\alpha)\cap S \not=\es$. Hence we can recursively
 enlarge every set $B^{<\alpha}_\xi$ adding to it a set $x_\xi+R_\alpha$ for some $x_\xi\in\bbr$ and keeping
 all enlarged sets pairwise disjoint -- it is enough to fulfill (5). To fulfill (4)  we have to enlarge our sets once
 more adding recursively to each of them one point from every set $P\in\PP_\alpha$.
 Again, we can do this without losing disjointness. As a result we obtain a family $\set{B^\alpha_\xi:\xi<c_\alpha}$
 which fulfills conditions (2)--(5). But the condition (1) is also fulfilled because constructing every set
 $B^\alpha_\xi$ we have added $|c_\alpha|$ many new points.

  Finally, we put
   $$B_\xi=\Sum_{\alpha<\kappa}B^\alpha_\xi$$
  (assuming that $B^\alpha_\xi=\es$ for $\alpha<\min\set{\eta:\xi<c_\eta}$).

  Thanks to (2) the family $\set{B_\xi:\xi<\cont}$ consists of pairwise disjoint sets and without
  problems we can extend them to get a partition of $\bbr$. By (4) every set $B_\xi$ is a Bernstein set.
  Moreover, the condition (5) is enough to ensure that every set $B_\xi$ is a $<\!\kappa$-covering. It is because
  every subset of the real line of cardinality smaller than $\kappa$ is a subset of one of the $R_\alpha$'s.
\end{proof}

%As an immediate corollary we obtain a result which is an irregular counterpart of Nowik's theorem
%on existence of a partition of the Cantor set into continuum many Borel $\omega$-coverings.

%\begin{cor}
% There exists a partition $\set{B_\alpha:\alpha<\cont}$ of the real line into continuum many Bernstein sets
% such that for every $\alpha<\cont$ the set $B_\alpha$ is an $\omega$-covering.
%\end{cor}

%Adding some set-theoretical assumptions we can derive from Theorem \ref{cofi} a stronger corollary.

%\begin{cor}
% If $\cont$ is a regular cardinal number then there exists a partition $\set{B_\alpha:\alpha<\cont}$ of
% the real line into continuum many Bernstein sets such that for every $\alpha<\cont$ the set $B_\alpha$
% is a $\,<\!\cont$-covering.
%\end{cor}

%Let us notice that we can obtain a result about $\omega_1$-coverings. Namely, we have

%\begin{cor}\label{cofiomega1}
% If $\cf(\cont)>\omega_1$ then there exists a partition $\set{B_\alpha:\alpha<\cont}$ of the real line into
% continuum many Bernstein sets such that for every $\alpha<\cont$ the set $B_\alpha$ is an $\omega_1$-covering.
%\end{cor}

On the other hand, as the only $\cont$-covering subset of the real line is the set $\bbr$ itself, we have the
following fact.

\begin{fact}
 Assume $CH$. Then there is no Bernstein set which is an $\omega_1$-covering.
\end{fact}

Now, one can pose the following question.

\begin{quest}\label{pyt}
 Assume $\cont>\omega_1=\cf(\cont).$ Is it true that there exists an $\omega_1$-covering which is a Bernstein set?
\end{quest}

It is worth mentioning that in the proof of Theorem~\ref{cofi} 
%(and hence Corollary~\ref{cofiomega1}) 
we have
succeeded in constructing relevant $\omega_1$-coverings because we have been able to cover every set of size
$\omega_1$ by a set of size smaller then continuum, taken from the fixed family of size at most continuum. Let us
notice that it is not possible to answer Question~\ref{pyt} using the similar method as in the proof of
Theorem~\ref{cofi} since we have the following observation which is a special case of the fact $cov(\lambda,\lambda,cf(\lambda)^+,2)>\lambda$ (see \cite{shelah}).

\begin{thm}[see \cite{shelah}]
 Assume that $\cont=\omega_{\omega_1}.$ Then there is no family $\BB\sbe[\bbr]^{<\cont}$ of size continuum such that
 every subset of $\bbr$ of size $\omega_1$ is covered by some set from the family $\BB$.
\end{thm}
%\begin{proof}
% Assume that such a family $\BB\sbe[\bbr]^{<\cont}$ exists and let us fix its enumeration
% $\BB=\set{B_\xi:\xi<\cont}$. Without loss of generality we may assume that if $\alpha<\omega_1$ then
% $|B_\beta|<\omega_\alpha$ for every $\beta<\omega_\alpha$. Now, using a standard diagonal argument
% we can choose for every $\alpha<\omega_1$ a real number $x_\alpha$ such that
%  $$x_\alpha\notin\Sum_{\beta<\omega_\alpha}B_\beta.$$
% Let us put $X=\{x_\alpha\ :\ \alpha<\omega_1\}$. There is no $\xi<\cont$ such that $X\sbe B_\xi$ what gives us a
% contradiction.
%\end{proof}

If we deal with completely $\II$-nonmeasurable sets instead of Bernstein sets then we can construct even a
$\,<\!\cont$-covering on condition the $\sid$ $\II$ has the Steinhaus property and its uniformity is not too big.

%\begin{defn}
%We say that the $\sid$ $\II\sbe\PP(\bbr)$ has \emph{the Steinhaus property} if for every sets $A\in\PP(\bbr)\sm\II$
%and $B\in\Borel(\bbr)\sm\II$ the set $A-B=\set{a-b:a\in A\,\land\,b\in B}$ contains an interval.
%\end{defn}

It is known that the $\sid$ of null sets and the $\sid$ of meager sets have the Steinhaus property (even in more
general context -- cf.~\cite{BCS},~\cite{McShane}).

\begin{prop}
 Assume that $\II\sbe\PP(\bbr)$ is a $\sidi$ having the Steinhaus property and such that $\non(\II)<\cont$. Then there
 exists a  $\,<\hskip-2pt\cont$-covering which is completely $\II$-nonmeasurable.
\end{prop}
\begin{proof}
Let us fix a set $N\notin\II$ such that $|N|=\non(\II)$ and put $C=\bbr\sm(N+\bbq)$. Suppose now that
$B\in\Borel(\bbr)\sm\II$. Then from the Steinhaus property of $\II$ we obtain that there exists a rational $q\in\bbq$
such that $q\in (\bbr\sm C)-B$. Hence $B\cap(\bbr\sm C)\ne\es$. As $|\bbr\sm C|<\cont$ we have also $B\cap C\ne\es$,
so the set $C$ is completely $\II$-nonmeasurable.

Moreover, the set $C$ is a $\,<\hskip-3pt\cont$-covering. Indeed, suppose that there exists a set
$A\in[\bbr]^{<\cont}$ such that for every $x\in\bbr$ we obtain $(A+x)\cap(\bbr\sm C)\not=\es$. For every $x\in\bbr$
let us fix $a_x\in A$ such that $a_x+x\in\bbr\sm C$. As $|\bbr\sm C|<\cont$ then there exist a set $T\sbe\bbr$ of
size continuum and a real $c\in\bbr\sm C$ such that for every $x\in T$ we have $a_x+x=c$. But it means that all reals
$a_x=c-x$ are different. Thus $|A|=|T|=\cont$ which is a contradiction.
\end{proof}

\section{S-coverings}

We can interpret $\kappa$-coverings in terms of coloring sets. Namely, we can treat a $\kappa$-covering
as set which can color every set of size $\kappa$ monochromatically. From this point of view we may ask
about a family of sets which can color every set of size $\kappa$ in such a way that different points in the
given set have different colors. This leads us to the following definition.

\begin{defn}\label{scovdef}
 A family $\AA$ of pairwise disjoint subsets of the real line is called a
 \emph{$\kappa$-S-covering} if $|\AA|=\kappa$ and
 $$(\fa F\in [\bbr ]^\kappa)(\ex t\in\bbr)\Big(F+t\sbe\Sum\AA\,\land\,(\fa A\in\AA)|(F+t)\cap A|=1\Big).$$
\end{defn}

\noindent
This definition is reasonable also for other uncountable abelian Polish groups.

First we prove a relation between 2-S-coverings and 2-coverings.

\begin{thm}\label{2scov}
 Assume that $\set{A_0,A_1}$ is a partition of the real line and a $2$-S-covering. Then at least one of the sets
 $A_0,A_1$ is a $2$-covering.
\end{thm}
\begin{proof}
 Assume that none of the sets $A_0,A_1$ is a 2-covering. It means that there are positive reals $a,b$ such that
 for every $x,y\in A_0$ we have $x-y\not=a$ and for every $x,y\in A_1$ we have $x-y\not=b$.
 We will show that the set $\set{0,a+b}$ cannot be S-covered by $\set{A_0,A_1}$.

 Indeed, let us fix any $x\in A_0$. Then $x+a\in A_1$ and, consequently, $x+a+b\in A_0$. Analogously, if
 $x\in A_1$ then $x+b+a\in A_1$, which ends the proof.
\end{proof}

Now we focus our attention on constructing $\kappa$-S-coverings which consist of Bernstein sets or completely
$\II$-nonmeasurable sets and such that none of their elements is a $\kappa$-covering (which is opposite to the
situation from Theorem~\ref{2scov}).

\begin{thm}\label{scover-partycja}
 Let $\kappa$ be a cardinal number such that $2<\kappa<\cont$. If $2^\kappa\le\cont$ then there exists a partition
 $\set{B_\xi:\xi<\kappa}$ of the real line such that
 \begin{enumerate}
  \item $(\fa\xi<\kappa)\,B_\xi$ is a Bernstein set,
  \item $(\fa\xi<\kappa)\,B_\xi$ is not a $2$-covering,
  \item $\set{B_\xi:\xi<\kappa}$ is a $\kappa$-S-covering.
 \end{enumerate}
\end{thm}
\begin{proof}
  Let $\Perf(\bbr)=\set{P_\alpha:\alpha<\cont}$ and $\bbr=\set{r_\alpha:\alpha<\cont}$ be fixed enumerations
  of all perfect subsets of the real line and of the reals, respectively. Let us also enumerate the set
  $[\bbr]^\kappa=\set{F_\alpha:\alpha<\cont}$.
  By transfinite induction we build a sequence $(\set{A^\alpha_\xi:\xi<\kappa})_{\alpha<\cont}$ of families of
  subsets of $\bbr$ of size less than continuum such that for every $\alpha<\cont$ the following conditions are
  fulfilled:
 \begin{enumerate}
  \item for every different $\xi_1,\xi_2<\kappa$ the sets $A^\alpha_{\xi_1}$ and $A^\alpha_{\xi_2}$ are disjoint;
  \item for every $\xi<\kappa$ the intersection $A^\alpha_\xi\cap P_\alpha$ is nonempty;
  \item for every $\xi<\kappa$ there exists $t_\alpha\in\bbr$ such that $|(t_\alpha+F_\alpha)\cap A^\alpha_\xi|=1$ and $t_\alpha+F_\alpha\subseteq\bigcup_{\xi\in\kappa}A^\alpha_\xi$;
  \item there exists $\xi<\kappa$ such that $r_\alpha\in A^\alpha_\xi$;
  \item for every $\xi<\kappa$ and every $\beta<\alpha$ we have $A^\beta_\xi\sbe A^\alpha_\xi$;
  \item for every $\xi<\kappa$ and every $x,y\in A^\alpha_\xi$ we have $|x-y|\ne 1$;
  \item $|A^\alpha_\xi|\le|\alpha|\cdot\omega.$
 \end{enumerate}
Suppose that we have already constructed the sequence $(\set{A^\beta_\xi:\xi<\kappa})_{\beta<\alpha}$ for some
$\alpha<\cont$. Let $A_\xi=\Sum_{\beta<\alpha} A^\beta_\xi$ and $A=\Sum_{\xi<\kappa} A_\xi$. We can observe that
there are not many "bad translations" of the set $F_\alpha$, namely the set
 $$T=\set{t\in\bbr:(\ex x\in F_\alpha)(\ex a\in A)\,|t+x-a|=1\, \lor\, t+x=a}$$
has the cardinality less then $\cont$. Thus we can choose a real $t_\alpha\notin T$.
Next we choose a subset $Y\sbe P_\alpha$ of size $\kappa$ such that
 $$(Y+\set{0,1,-1})\cap ((t_\alpha+F_\alpha)\cup A)=\es.$$
Let $\set{a_\xi:\xi<\kappa}$ and $\set{b_\xi:\xi<\kappa}$ be enumerations of sets $t_\alpha+F_\alpha$ and $Y$,
respectively, and let $\hat{A}^\alpha_\xi=A_\xi\cup\set{a_\xi,b_\xi}$ for $\xi<\kappa$. Finally, if
$r_\alpha\notin Y\cup(t_\alpha+F_\alpha)\cup A$ then we fix $\xi_0<\kappa$ such that
$\hat{A}^\alpha_{\xi_0^{\ }}\cap\set{r_\alpha-1,r_\alpha+1}=\es$ and put
$A^\alpha_{\xi_0^{\ }}=\hat{A}^\alpha_{\xi_0^{\ }}\cup\set{r_\alpha}$. In all other cases we put
$A^\alpha_{\xi}=\hat{A}^\alpha_{\xi}$ and our construction is completed.

Let $B_\xi=\Sum_{\alpha<\cont}A_\xi^\alpha$ for $\xi<\kappa$. Then $B_\xi$ is a Bernstein set thanks to
the condition (2) and is not a 2-covering thanks to the conditions (5) and (6). The conditions (1) and (4) ensure us
that the family $\set{B_\xi:\xi<\kappa}$ is a partition of $\bbr$ and the condition (3) makes this family
a $\kappa$-S-covering.
\end{proof}

\begin{rem}
Let us observe that if $\kappa$ is countable then the condition $2^\kappa\le\cont$ is fulfilled. In general
we need extra set theoretic assumptions. For example it is enough to assume Martin's Axiom, which implies that
$2^\kappa=\cont$ for $\omega\le\kappa<\cont$ (see~\cite{jech}).
\end{rem}

In more general situation, constructing S-coverings consisting of completely $\II$-nonmeasurable subsets of a given
Polish group, none of which is a 2-covering is a bit more complicated. That is why we need some additional
assumptions about a $\sid$ $\II$.

\begin{thm}\label{scover_grupa}
Let $(X,+)$ be an uncountable abelian Polish group with a complete metric $d$. Let $\II\sbe\PP(X)$ be a $\sidi$ such
that 
  $$(\forall B\in\Borel(X)\setminus\II)(\fa \DD\in [\II]^{<\cont})\, |B\sm \Sum\DD|=\cont$$
and there exists $a\in rng(d),\ a\neq 0$ such that
  $$(\fa x\in X)\,\set{y\in X:d(x,y)=a}\in\II.$$
If $\kappa$ is a cardinal number such that $2^\kappa=\cont$, then there exists a family $\set{B_\xi:\xi<\kappa}$ of
pairwise disjoint subsets of $X$ such that
 \begin{enumerate}
  \item $(\fa\xi<\kappa)\,B_\xi$ is a completely $\II$-nonmeasurable set,
  \item $(\fa\xi<\kappa)\,B_\xi$ is not a $2$-covering,
  \item $\set{B_\xi:\xi<\kappa}$ is a $\kappa$-S-covering.
 \end{enumerate}
\end{thm}
\begin{proof}
 Without loss of generosity $a=1.$
 Let $\Borel(X)\sm\II=\set{P_\alpha:\alpha <\cont}$ be an enumeration of all $\II$-positive Borel subsets of $X$.
 Let us also enumerate the set $[X]^\kappa=\set{F_\alpha:\alpha<\cont}$. We proceed similarly as in the proof of
 Theorem~\ref{scover-partycja}, constructing a sequence $(\set{A^\alpha_\xi:\xi<\kappa})_{\alpha<\cont}$ of
 families of subsets of $X$ of size less than continuum such that for every $\alpha<\cont$ the following
 conditions are fulfilled:
 \begin{enumerate}
  \item for every different $\xi_1,\xi_2<\kappa$ the sets $A^\alpha_{\xi_1}$ and $A^\alpha_{\xi_2}$ are disjoint;
  \item for every $\xi<\kappa$ the intersection $A^\alpha_\xi\cap P_\alpha$ is nonempty;
  \item for every $\xi<\kappa$ there exists $t_\alpha\in X$ such that $|(t_\alpha+F_\alpha)\cap A^\alpha_\xi|=1$ and $t_\alpha+ F_\alpha\subseteq\bigcup_{\xi<\kappa}A^\alpha_\xi$;
  \item for every $\xi<\kappa$ and every $\beta<\alpha$ we have $A^\beta_\xi\sbe A^\alpha_\xi$ and $|A^\alpha_\xi|\le|\alpha|\cdot\omega$;
  \item for every $\xi<\kappa$ and every $x,y\in A^\alpha_\xi$ we have $d(x,y)\ne 1$.
 \end{enumerate}
Assume that we are at an $\alpha$th step of the construction and let $A_\xi=\Sum_{\beta<\alpha} A^\beta_\xi$ and
$A=\Sum_{\xi<\kappa} A_\xi$. Moreover, let $C=\Sum_{x\in F_\alpha}\Sum_{a\in A}\set{t\in X:d(t+x,a)=1}$.
Then the set $T=C\cup(A-F_\alpha)$ is the set of "bad translations" of the set $F_\alpha$. But $C$ is a
collection of less then continuum many unit spheres and $|A-F_\alpha|<\cont$ so according to our assumptions
the complement of $T$ is of size continuum. Thus we can choose $t_\alpha\notin T$.

Analogously, we can choose a subset $Y\sbe P_\alpha$ of size $\kappa$ such that
 $$Y\cap((t_\alpha+F_\alpha)\cup A\cup\set{x\in X:(\ex a\in(t_\alpha+F_\alpha)\cup A)\,d(x,a)=1})=\es.$$
Finally, we enumerate sets $t_\alpha+F_\alpha=\set{a_\xi:\xi<\kappa}$ and $Y=\set{b_\xi:\xi<\kappa}$, put
$A^\alpha_\xi=A_\xi\cup\set{a_\xi,b_\xi}$ for $\xi<\kappa$ and we are done.

Let $B_\xi=\Sum_{\alpha<\cont}A_\xi^\alpha$ for $\xi<\kappa$. Then $\set{B_\xi:\xi<\kappa}$
is the needed family.
\end{proof}

\begin{rem}\label{rem_add}
Let us observe that in Theorem~\ref{scover_grupa} we can replace the assumption
 $$(\forall B\in\Borel(X)\setminus\II)(\fa \DD\in [\II]^{<\cont})\, |B\sm \Sum\DD|=\cont$$
by a stronger, but shorter assumption, namely $\add(\II)=\cont$.
\end{rem}

When our Polish space is simply a Euclidian vector space, we can omit one assumption in
Theorem~\ref{scover_grupa}.

\begin{cor}\label{cor_grupa}
Let $\II\sbe\PP(\bbr^n)$ be a $\sidi$ containing all unit spheres. Then for every cardinal number
$\kappa$ such that $2^\kappa=\cont$ there exists a family $\set{B_\xi:\xi<\kappa}$ of pairwise
disjoint subsets of $X$ such that
 \begin{enumerate}
  \item $(\fa\xi<\kappa)\,B_\xi$ is a completely $\II$-nonmeasurable set,
  \item $(\fa\xi<\kappa)\,B_\xi$ is not a $2$-covering,
  \item $\set{B_\xi:\xi<\kappa}$ is a $\kappa$-S-covering.
 \end{enumerate}
\end{cor}
\begin{proof}
It is enough to observe that we can repeat the proof of Theorem~\ref{scover_grupa}.
Indeed, our choice of $Y$ (and $t_\alpha$) is possible because thanks to Lemma~\ref{covering2} after
removing less than continuum many unit spheres from an $\II$-positive Borel set we have still continuum
many points left.
\end{proof}

Just as in the case of Theorem~\ref{scover-partycja}, assuming Martin's Axiom we obtain from
Theorem~\ref{scover_grupa} a suitable $\kappa$-S-covering for every $\kappa<\cont$. However, it occurs
that we can do this uniformly on condition $\cont$ is regular.

\begin{defn}\label{scovdef}
 A family $\AA$ of pairwise disjoint subsets of an uncountable abelian Polish group $(X,+)$ is called a
 \emph{$\,<\!\kappa$-S-covering}
 $$(\fa F\in [X]^{<\kappa})(\ex t\in X)\Big(F+t\sbe\Sum\AA\,\land\,(\fa A\in\AA)|(F+t)\cap A|\le 1\Big).$$
\end{defn}

\begin{thm}\label{scover_martin_axiom}
Let $(X,+)$ be an uncountable abelian Polish group with a complete metric $d$. Let $\II\sbe\PP(X)$
be a $\sidi$ such that
  $$(\forall B\in\Borel(X)\setminus\II)(\fa \DD\in [\II]^{<\cont})\, |B\sm \Sum\DD|=\cont$$
and there exists $a\in rng(d),\ a\neq 0$ such that
  $$(\fa x\in X)\,\set{y\in X:d(x,y)=a}\in\II.$$
If for every $\kappa<\cont$ we have $2^\kappa\le\cont$
then there exists a family $\set{B_\xi:\xi<\cont}$ of pairwise disjoint subsets of $X$ such that
 \begin{enumerate}
  \item $(\fa\xi<\cont)\,B_\xi$ is a completely $\II$-nonmeasurable set,
  \item $(\fa\xi<\cont)\,B_\xi$ is not a $2$-covering,
  \item $\set{B_\xi:\xi<\cont}$ is a $\,<\!\cont$-S-covering.
 \end{enumerate}
\end{thm}

\begin{proof}
 Let $\Borel(X)\sm\II=\set{P_\alpha:\alpha<\cont}$ be an enumeration of all $\II$-positive Borel subsets of $X$.
 We also enumerate the set $[X]^{<\cont}=\set{F_\alpha:\alpha<\cont}$ in such way that every $F\in [X]^{<\cont}$
 appears in this enumeration cofinally often.
 By transfinite induction we construct a matrix $(A^\eta_\xi)_{\xi,\eta<\cont}$ of subsets of $X$ of
 size less than continuum such that for every $\alpha<\cont$ the following conditions are fulfilled:
 \begin{enumerate}
  \item for every $\eta\le\alpha$ and every different $\xi_1,\xi_2\le\alpha$ the sets
  $A^\eta_{\xi_1}$ and $A^\eta_{\xi_2}$ are disjoint;
  \item for every $\xi,\eta\le\alpha$ the intersection $A^\eta_\xi\cap P_\eta$ is nonempty;
  \item for every $\xi\le\alpha$ and every $\eta_1<\eta_2\le\alpha$ we have
  $A^{\eta_1^{\ }}_\xi\sbe A^{\eta_2^{\ }}_\xi$;
  \item for every $\xi,\eta\le\alpha$ and every $x,y\in A^\eta_\xi$ we have $d(x,y)\ne 1$;
  \item if $|F_\alpha|\le\alpha$ then there exists $t_\alpha\in X$ such that
  $t_\alpha+F_\alpha\sbe\Sum_{\xi\le\alpha}A^\alpha_\xi$ and for every $\xi\le\alpha$ we have
  $|(t_\alpha+F_\alpha)\cap A^\alpha_\xi|\le 1$.
 \end{enumerate}
Our construction is similar to this from the proof of Theorem~\ref{scover_grupa}. Adding new points
we have to take care that they are different from the ones constructed before and that they do not belong to any
unit sphere with a center in an old point, which is always possible because of the assumptions.

%Assume that we are at an $\alpha$th step of our construction and we have already built the matrix
%$(A^\eta_\xi)_{\xi,\eta<\alpha}$ for some $\alpha<\cont$. First, we fix step by step %elements $a_\eta\in P_\eta$ 
%for $\eta<\alpha$. Then we put $A^\eta_\alpha=\set{a_\beta:\beta\le\eta}$ and denote 
%$A_\xi=\Sum_{\eta<\alpha}A^\eta_\xi$. Next, we have to consider two cases.

%If $\alpha<|F_\alpha|$ then we fix a set $\set{b_\xi:\xi\le\alpha}\sbe P_\alpha$ and
%put $A^\alpha_\xi=A_\xi\cup\set{b_\xi}$ for $\xi\le\alpha$. If $\alpha\ge|F_\alpha|=\mu$
%then, as in the proof of Theorem~\ref{scover_grupa} we observe that the set of "bad translations" of
%the set $F_\alpha$ has complement of size continuum, hence we can fix a "good translation" $t_\alpha\in X$.
%Let us enumerate the set $t_\alpha+F_\alpha=\set{a_\xi:\xi<\mu}$. Next, we find a suitable set 
%$\set{b_\xi:\xi\le\alpha}\sbe P_\alpha$ and define
%  $$A^\alpha_\xi=\left\{
%    \begin{array}{ll}
%    A_\xi\cup\set{a_\xi,b_\xi} & \mbox{for}\ \xi<\mu\\
%    A_\xi\cup\set{b_\xi} & \mbox{for}\ \mu\le\xi<\alpha\\
%    \end{array}\right.$$
%for $\xi\le\alpha$.    

Let $B_\xi=\Sum_{\xi\le\alpha<\cont}A^\alpha_\xi$ for $\xi<\cont$. As in the previous theorems, the sets $B_\xi$
are pairwise disjoint and none of them is a 2-covering. Moreover, for every $F\in [X]^{<\cont}$ there exists
$\alpha<\cont$ such that $F=F_\alpha$ and $|F_\alpha|\le\alpha$, so by the condition (5) there exists $t=t_\alpha$
being a witness for that the family $\set{B_\xi:\xi<\cont}$ is a $\,<\!\cont$-S-covering. Finally,
by the condition (2) every set $B_\xi$ intersects all $\II$-positive Borel subsets of $X$. Hence the set $B_\xi$ is 
completely $\II$-nonmeasurable for any $\xi<\cont$.
\end{proof}

As a corollary we obtain a result concerning an S-covering made of Lebesgue completely nonmeasurable sets in $\bbr^n.$

\begin{cor}
Assume Martin's Axiom and $\cont=\aleph_2$. Then there exists a family $\set{B_\xi:\xi<\cont}$ of pairwise
disjoint subsets of $\,\bbr^n$ such that
 \begin{enumerate}
  \item $(\fa\xi<\cont)\,\lambda_*(B_\xi)=0$ and $\lambda_*(\bbr^n\sm B_\xi)=0$,
  \item $(\fa\xi<\cont)\,B_\xi$ is not a $2$-covering,
  \item $\set{B_\xi:\xi<\cont}$ is a $\omega_1$-S-covering,
 \end{enumerate}
where $\lambda_*$ denotes the inner Lebesgue measure in $\bbr^n$.
\end{cor}
\begin{proof}
Immediate from Theorem~\ref{scover_martin_axiom}, Corollary~\ref{cor_grupa} and Remark~\ref{rem_add} together
with the fact that under Martin's Axiom the additivity of the $\sid$ of Lebesgue null sets is equal to
continuum.
\end{proof}

\section{I-coverings on the plane}

In this chapter we focus our attention on the plane $\bbr^2$ treated as a Polish group. According to
Definition~\ref{covdef} we can investigate a $\kappa$-covering as a subset of the plane such that every
planary set of size $\kappa$ can be translated into it. However, we may also generalize this definition
letting sets of size $\kappa$ to be not only translated but moved by any isometry.

\begin{defn}
 We say that a set $A\sbe\bbr^2$ is a $\kappa$-I-covering if
 $$(\fa B\in[\bbr^2]^\kappa)(\ex\phi:\bbr^2\rightarrow\bbr^2)
   (\phi \mbox{ is an isometry and } \phi[B]\sbe A).$$
\end{defn}

It occurs that we cannot partition the plane into two sets none of which is a 2-I-covering.

\begin{thm}
 If $\set{A_0,A_1}$ is a partition of $\bbr^2$ then one of the sets $A_0,A_1$ is a $2$-I-covering.
\end{thm}

\begin{proof}
Suppose that $A_0$ is not a 2-I-covering. Then there exists a positive real $d$ such that none two points
in $A_0$ are at a distance of $d$ from each other. Let us fix any $a\in A_0$ and consider a circle $C$ with a center
$a$ and a radius equal to $d$. Next, let us fix a halfline that starts from $a$ and consider such a sequence
$(a_n)_{n<\omega}$ of elements of this halfline that $d(a,a_n)=(n+2)d$ for all $n<\omega$. Then for every real
$x\in[(n+1)d,(n+3)d]$ there exists a point $p\in C$ such that $d(p,a_n)=x$.

Observe now that $C\sbe A_1$. Moreover, at least one of every two consecutive elements of the sequence
$(a_n)_{n<\omega}$ belongs to $A_1$. Hence for every $x>0$ we can find two elements of $A_1$ which
are at a distance of $x$ from each other. Consequently, the set $A_1$ is a 2-I-covering.
\end{proof}

Next two theorems show that from the point of view of Bernstein sets there is a big difference between
2-I-coverings and 3-I-coverings.

\begin{thm}\label{2icov}
Every Bernstein set is a $2$-I-covering.
\end{thm}

\begin{proof}
Let $B\sbe\bbr^2$ be a Bernstein set. To show that $B$ is also a $2$-I-covering let us fix two different
points $a,b\in\bbr^2$. It is enough to observe that any circle with a center in a fixed point $c\in B$
and a radius $d(a,b)$ (where $d$ stands for a standard Euclidean metric) is a perfect set, thus meets $B$.
\end{proof}

\begin{thm}
There exists a Bernstein set which is not a $3$-I-covering.
\end{thm}
\begin{proof}
Let $\Perf(\bbr^2)=\set{P_\alpha:\alpha<\cont}$ be a fixed enumeration of all perfect subsets of $\bbr^2$.
We build by transfinite induction two sequences $(a_\alpha)_{\alpha<\cont},(b_\alpha)_{\alpha<\cont}$ of
elements of the plane satisfying the following conditions:
 \begin{enumerate}
    \item $(\fa \alpha<\cont)\,a_\alpha,b_\alpha \in P_\alpha$,
    \item $\set{a_\alpha:\alpha<\cont}\cap\set{b_\alpha:\alpha<\cont}=\es$,
    \item $(\fa \alpha,\beta,\gamma<\cont)(d(a_\alpha,a_\beta)\ne 1\,\lor\,d(a_\alpha,a_\gamma)\ne 1\,
    \lor\,d(a_\beta,a_\gamma)\ne 1)$.
 \end{enumerate}
Suppose that we have already constructed $(a_\xi)_{\xi<\alpha}$ and $(b_\xi)_{\xi<\alpha}$ for some $\alpha<\cont$.
Since the set $A=\set{(a_{\xi_1}, a_{\xi_2}):\xi_1,\xi_2<\alpha\,\land\,d(a_{\xi_1},a_{\xi_2})= 1}$ has at most
$|\alpha\times\alpha|<\cont$ elements and for every pair $(a_{\xi_1}, a_{\xi_2})\in A$  there are only two points
with distance 1 from both $a_{\xi_1}$ and $a_{\xi_2}$ we can pick
$a_\alpha\in P_\alpha\sm(\set{a_\xi: \xi<\alpha}\cup\set{b_\xi: \xi<\alpha})$ such that
$d(a_\alpha,a_{\xi_1})\neq 1$ or $d(a_\alpha,a_{\xi_2})\neq 1$ for all $\xi_1,\xi_2<\alpha$. Let $b_\alpha$ be any
element of $P_\alpha\sm(\set{a_\xi:\xi\le\alpha}\cup\set{b_\xi:\xi<\alpha})$.

Let us put $B=\set{a_\alpha:\alpha<\cont}$. The condition (2) ensures $B$ is a Bernstein set. To show that $B$
is not a $3$-I-covering it is enough to observe that there is no equilateral triangle of sides of length 1 with
all vertices in $B$.
\end{proof}

When we replace Bernstein sets by completely $\II$-nonmeasurable sets then it occurs that the theorem analogous
to Theorem~\ref{2icov} may not be true.

\begin{thm}
Let $\II\sbe\PP(\bbr^2)$ be a $\sidi$ such that every unit circle is in $\II$.
Then there exists a completely $\II$-nonmeasurable set which is not a $2$-I-covering.
\end{thm}
\begin{proof}
Let $\Borel(X)\sm\II=\set{B_\alpha:\alpha<\cont}$ be an enumeration of all $\II$-positive Borel subsets of $X$.
We build by transfinite induction two sequences $(a_\alpha)_{\alpha<\cont},(b_\alpha)_{\alpha<\cont}$ of
elements of the plane satisfying the following conditions:
 \begin{enumerate}
    \item $(\fa \alpha<\cont)\,a_\alpha,b_\alpha \in B_\alpha$,
    \item $\set{a_\alpha:\alpha<\cont}\cap\set{b_\alpha:\alpha<\cont}=\es$,
    \item $(\fa \alpha,\beta<\cont)\,d(a_\alpha,a_\beta)\ne 1$.
 \end{enumerate}
Assume that we are at an $\alpha$th step of the construction. Let
 $$D=B_\alpha\sm\Sum_{\beta<\alpha}\set{a\in\bbr^2: d(a,a_\beta)=1}.$$
From Lemma~\ref{covering2} we get $|D|=\cont$. Let us pick $a_\alpha\in D\sm\set{a_\beta:\beta<\alpha}$ and
let $b_\alpha\in B_\alpha\sm(\set{a_\beta:\beta\le\alpha}\cup\set{b_\beta:\beta<\alpha})$.

Finally, the set $B=\set{a_\alpha:\alpha<\cont}$ is completely $\II$-nonmeasurable and not a $2$-I-covering.
\end{proof}

\end{document}